\documentclass[a4paper,12pt]{article}
\usepackage{times, url}
\usepackage{amsmath,amssymb,amsthm,amsfonts}

\newtheorem{theorem}{Theorem}
\newtheorem{lemma}{Lemma}

\begin{document}
\setcounter{page}{1}

$\vphantom{A}$

$\vphantom{A}$

$\vphantom{A}$

$\vphantom{A}$

$\vphantom{A}$

$\vphantom{A}$

\begin{center}
{\LARGE \bf Lower bounds on expressions dependent on functions {\boldmath$\varphi(n)$}, {\boldmath$\psi(n)$} and {\boldmath$\sigma(n)$}, II}
\vspace{8mm}

{\Large \bf Stoyan Dimitrov}
\vspace{3mm}

Faculty of Applied Mathematics and Informatics, Technical University of Sofia\\
Blvd. St.Kliment Ohridski 8, Sofia 1756, Bulgaria \\
e-mail: \url{sdimitrov@tu-sofia.bg}

\vspace{2mm}

Department of Bioinformatics and Mathematical Modelling \\
Institute of Biophysics and Biomedical Engineering\\
Bulgarian Academy of Sciences\\
Acad. G. Bonchev Str. Bl. 105, Sofia 1113, Bulgaria \\
e-mail: \url{ xyzstoyan@gmail.com}
\vspace{2mm}
\end{center}
\vspace{10mm}

{\bf Abstract:}
In this paper we establish lower bounds on several expressions dependent on functions $\varphi(n)$, $\psi(n)$ and $\sigma(n)$.\\
{\bf Keywords:} Arithmetic functions $\varphi(n)$, $\psi(n)$ and $\sigma(n)$, Lower bounds.\\
{\bf 2020 Mathematics Subject Classification:}  11A25.
\vspace{10mm}

\section{Notations and formulas}
\indent

The letter $p$ with or without subscript will always denote prime number.
Let $n>1$ be positive integer with prime factorization
\begin{equation*}
n=p^{a_1}_1\cdots p^{a_k}_k\,.
\end{equation*}
The function $\Omega(n)$ counts the total number of prime factors of $n$ honoring their multiplicity. We have
\begin{equation*}
\Omega(n)=\sum\limits_{i=1}^{k}a_i \quad \mbox{ and } \quad \Omega(1)=0\,.
\end{equation*}
We denote by $\varphi(n)$ the Euler totient function which is defined as
the number of positive integers not greater than $n$ that are coprime to $n$.
We have
\begin{equation*}
\varphi(n)=\prod\limits_{i=1}^{k}p^{a_i-1}_i(p_i-1) \quad \mbox{ and } \quad \varphi(1)=1\,.
\end{equation*}
We define the Dedekind function $\psi(n)$ by the formula
\begin{equation*}
\psi(n)=\prod\limits_{i=1}^{k}p^{a_i-1}_i(p_i+1) \quad \mbox{ and } \quad \psi(1)=1\,.
\end{equation*}
The function $\sigma(n)$ denotes the sum of the positive divisors of $n$.
We have
\begin{equation*}
\sigma(n)=\prod\limits_{i=1}^{k}\frac{p^{a_i+1}_i-1}{p_i-1} \quad \mbox{ and } \quad \sigma(1)=1\,.
\end{equation*}

\section{Introduction and statement of the results}
\indent

In 2013 Atanassov \cite{Atanassov2013} proved that for every natural number $n\geq2$ the lower bound
\begin{equation*}
\varphi(n)\psi(n)\sigma(n)\geq n^3+n^2-n-1
\end{equation*}
holds. Afterwards S\'{a}ndor \cite{Sandor2014} improved Atanassov's result proving that
for all $n\geq1$ one has the inequalities
\begin{equation*}
\varphi(n)\psi(n)\sigma(n)\geq \varphi^\ast(n)(\sigma^\ast(n))^2\geq n^3+n^2-n-1\,,
\end{equation*}
where $\psi^\ast(n)$ and $\sigma^\ast(n)$ are the unitary analogues of the functions $\psi(n)$ and $\sigma(n)$.
We refer to \cite{Sandor1990}, \cite{SandorHandbook} for definitions, properties and references.
Very recently the author \cite{Dimitrov} showed that
\begin{align*}
&\varphi^2(n)+\psi^2(n)+\sigma^2(n)\geq 3n^2+2n+3\,,\\
&\varphi(n)\psi(n)+\varphi(n)\sigma(n)+\sigma(n)\psi(n)\geq 3n^2+2n-1
\end{align*}
for every natural number $n\geq2$.
As a continuation of these studies, we establish the following five theorems.
\begin{theorem}
For every natural number $n\geq2$ the lower bound
\begin{equation}\label{lowerbound1}
\varphi^3(n)+\psi^3(n)+\sigma^3(n)\geq 3n^3+3n^2+9n+1
\end{equation}
holds.
\end{theorem}

\begin{theorem}
For every natural number $n\geq2$ the lower bound
\begin{equation}\label{lowerbound2}
\varphi^4(n)+\psi^4(n)+\sigma^4(n)\geq 3n^4+4n^3+18n^2+4n+3
\end{equation}
holds.
\end{theorem}

\begin{theorem}
For every natural number $n\geq2$ the lower bound
\begin{equation}\label{lowerbound3}
\varphi^2(n)\psi^2(n)+\varphi^2(n)\sigma^2(n)+\sigma^2(n)\psi^2(n)\geq 3n^4+4n^3+2n^2+4n+3
\end{equation}
holds.
\end{theorem}

\begin{theorem}
For every natural number $n\geq2$ the lower bound
\begin{equation}\label{lowerbound4}
\varphi^2(n)\big(\psi(n)+\sigma(n)\big)+\psi^2(n)\big(\varphi(n)+\sigma(n)\big)+\sigma^2(n)\big(\varphi(n)+\psi(n)\big)\geq 6n^3+6n^2+2n+2
\end{equation}
holds.
\end{theorem}

\begin{theorem}
For every natural number $n\geq2$ the lower bound
\begin{equation}\label{lowerbound5}
\varphi^3(n)\big(\psi(n)+\sigma(n)\big)+\psi^3(n)\big(\varphi(n)+\sigma(n)\big)+\sigma^3(n)\big(\varphi(n)+\psi(n)\big)\geq 6n^4+8n^3+12n^2+8n-2
\end{equation}
holds.
\end{theorem}

\section{Lemmas}
\indent

\begin{lemma}\label{varphipsi2n}
For every natural number $n\geq2$ the lower bound
\begin{equation*}
\varphi(n)+\psi(n)\geq2n
\end{equation*}
holds.
\end{lemma}
\begin{proof}
See (\cite{Atanassov2011}, Lemma 1).
\end{proof}

\begin{lemma}\label{varphisigma2n}
For every natural number $n\geq2$ the lower bound
\begin{equation*}
\varphi(n)+\sigma(n)\geq2n
\end{equation*}
holds.
\end{lemma}
\begin{proof}
See (\cite{Sandor2019}, Remark 5).
\end{proof}

\section{Proof of Theorem 1}
\indent

Consider several cases.

Case 1. \; $\Omega(n)=1$. Bearing in mind that $n$ is a prime number we write
\begin{equation*}
\varphi^3(n)+\psi^3(n)+\sigma^3(n)=(n-1)^3+2(n+1)^3=3n^3+3n^2+9n+1\,.
\end{equation*}

Case 2. \; $\Omega(n)=2$, $n=pq$, where $p$ and $q$ are distinct primes. Then
\begin{align*}
&\varphi^3(n)+\psi^3(n)+\sigma^3(n)\\
&=(p-1)^3(q-1)^3+2(p+1)^3(q+1)^3\\
&=3p^3q^3+27p^2q^2+27pq+3+3p^3q^2+3p^2q^3+9p^3q+9pq^3\\
&+p^3+9p^2q+9pq^2+q^3+9p^2+9q^2+3p+3q\\
&>3p^3q^3+3p^2q^2+9pq+1\\
&=3n^3+3n^2+9n+1\,.
\end{align*}

Case 3. \; $\Omega(n)=2$, $n=p^2$, where $p$ is a prime. Then
\begin{align*}
\varphi^3(n)+\psi^3(n)+\sigma^3(n)&=p^3(p-1)^3+p^3(p+1)^3+(p^2+p+1)^3\\
&=3p^6+3p^5+12p^4+7p^3+6p^2+3p+1\\
&>3p^6+3p^4+9p^2+1\\
&=3n^3+3n^2+9n+1\,.
\end{align*}
Now we assume that \eqref{lowerbound1} is true for every natural number $n$ with
$\Omega(n)=m$ for some natural number $m\geq2$.
Let $p$ be a prime number. Then $\Omega(np)=\Omega(n)+1$.

Case A.  \;  $p\nmid n$. Using that
\begin{equation}\label{varphin<n}
\varphi(n)<n
\end{equation}
we obtain
\begin{align*}
\varphi^3(np)+\psi^3(np)+\sigma^3(np)&=\varphi^3(n)(p-1)^3+\psi^3(n)(p+1)^3+\sigma^3(n)(p+1)^3\\
&=(p+1)^3\big[\varphi^3(n)+\psi^3(n)+\sigma^3(n)\big]-(6p^2+2)\varphi^3(n)\\
&\geq(p+1)^3(3n^3+3n^2+9n+1)-(6p^2+2)n^3\\
&=3n^3p^3+9n^2p^2+27np+1+p^3(3n^2+9n+1)\\
&+3p^2(n^3+9n+1)+3p(3n^3+3n^2+1)+n^3+3n^2+9n\\
&>3n^3p^3+3n^2p^2+9np+1\,.
\end{align*}

Case B.  \;  $p\,|\,n$. Using that
\begin{equation}\label{varphipsisigma}
\varphi(np)=p\varphi(n)\,, \quad \psi(np)=p\psi(n)\,, \quad \sigma(np)>p\sigma(n)
\end{equation}
we get
\begin{align*}
\varphi^3(np)+\psi^3(np)+\sigma^3(np)
&>p^3\big[\varphi^3(n)+\psi^3(n)+\sigma^3(n)\big]\\
&\geq p^3(3n^3+3n^2+9n+1)\\
&=3n^3p^3+3n^2p^2+9np+3n^2p^2(p-1)+9np(p^2-1)+p^3\\
&>3n^3p^3+3n^2p^2+9np+1\,.
\end{align*}
This completes the proof of Theorem 1.

\section{Proof of Theorem 2}
\indent

Consider several cases.

Case 1. \; $\Omega(n)=1$. Taking into account that $n$ is a prime number we have
\begin{equation*}
\varphi^4(n)+\psi^4(n)+\sigma^4(n)=(n-1)^4+2(n+1)^4=3n^4+4n^3+18n^2+4n+3\,.
\end{equation*}

Case 2. \; $\Omega(n)=2$, $n=pq$, where $p$ and $q$ are distinct primes. Then
\begin{align*}
&\varphi^4(n)+\psi^4(n)+\sigma^4(n)\\
&=(p-1)^4(q-1)^4+2(p+1)^4(q+1)^4\\
&=3p^4q^4+48p^3q^3+108p^2q^2+48pq+3+4p^4q^3+4p^3q^4+18p^4q^2+18p^2q^4\\
&+4p^4q+24p^3q^2+24p^2q^3+4pq^4+3p^4+48p^3q+48pq^3+3q^4\\
&+4p^3+24p^2q+24pq^2+4q^3+18p^2+18q^2+4p+4q\\
&>3p^4q^4+4p^3q^3+18p^2q^2+4pq+3\\
&=3n^4+4n^3+18n^2+4n+3\,.
\end{align*}

Case 3. \; $\Omega(n)=2$, $n=p^2$, where $p$ is a prime. Then
\begin{align*}
\varphi^4(n)+\psi^4(n)+\sigma^4(n)&=p^4(p-1)^4+p^4(p+1)^4+(p^2+p+1)^4\\
&=3p^8+4p^7+22p^6+16p^5+21p^4+16p^3+10p^2+4p+1\\
&>3p^8+4p^6+18p^4+4p^2+3\\
&=3n^4+4n^3+18n^2+4n+3\,.
\end{align*}
Now we assume that \eqref{lowerbound2} is true for every natural number $n$ with
$\Omega(n)=m$ for some natural number $m\geq2$.
Let $p$ be a prime number. Then $\Omega(np)=\Omega(n)+1$.

Case A.  \;  $p\nmid n$. Using \eqref{varphin<n} we find
\begin{align*}
\varphi^4(np)+\psi^4(np)+\sigma^4(np)&=\varphi^4(n)(p-1)^4+\psi^4(n)(p+1)^4+\sigma^4(n)(p+1)^4\\
&=(p+1)^4\big[\varphi^4(n)+\psi^4(n)+\sigma^4(n)\big]-(8p^3+8p)\varphi^4n)\\
&\geq(p+1)^4(3n^4+4n^3+18n^2+4n+3)-(8p^3+8p)n^4\\
&=3n^4p^4+16n^3p^3+108n^2p^2+16np+3+4n^4p^3+4n^3p^4\\
&+18n^4p^2+18n^2p^4+4n^4p+24n^3p^2+72n^2p^3+4np^4\\
&+3n^4+16n^3p+16np^3+3p^4+4n^3+72n^2p+24np^2+12p^3\\
&+18n^2+18p^2+4n+12p\\
&>3n^4p^4+4n^3p^3+18n^2p^2+4np+3\,.
\end{align*}

Case B.  \;  $p\,|\,n$. By \eqref{varphipsisigma} we obtain
\begin{align*}
\varphi^4(np)+\psi^4(np)+\sigma^4(np)
&>p^4\big[\varphi^4(n)+\psi^4(n)+\sigma^4(n)\big]\\
&\geq p^4(3n^4+4n^3+18n^2+4n+3)\\
&=3n^4p^4+4n^3p^3+18n^2p^2+4np+3\\
&+4n^3p^3(p-1)+18n^2p^2(p^2-1)+4np(p^3-1)+3(p^4-1)\\
&>3n^4p^4+4n^3p^3+18n^2p^2+4np+3\,.
\end{align*}
This completes the proof of Theorem 2.

\section{Proof of Theorem 3}
\indent

Consider several cases.

Case 1. \; $\Omega(n)=1$. Having in mind that $n$ is a prime number we deduce
\begin{equation*}
\varphi^2(n)\psi^2(n)+\varphi^2(n)\sigma^2(n)+\sigma^2(n)\psi^2(n)=2(n^2-1)^2+(n+1)^4=3n^4+4n^3+2n^2+4n+3\,.
\end{equation*}

Case 2. \; $\Omega(n)=2$, $n=pq$, where $p$ and $q$ are distinct primes. Then
\begin{align*}
&\varphi^2(n)\psi^2(n)+\varphi^2(n)\sigma^2(n)+\sigma^2(n)\psi^2(n)\\
&=2(p^2-1)^2(q^2-1)^2+(p+1)^4(q+1)^4\\
&=3p^4q^4+16p^3q^3+44p^2q^2+16pq+3+4p^4q^3+4p^3q^4+2p^4q^2+2p^2q^4\\
&+4p^4q+24p^3q^2+24p^2q^3+4pq^4+3p^4+16p^3q+16pq^3+3q^4\\
&+4p^3+24p^2q+24pq^2+4q^3+2p^2+2q^2+4p+4q\\
&>3p^4q^4+4p^3q^3+2p^2q^2+4pq+3\\
&=3n^4+4n^3+2n^2+4n+3\,.
\end{align*}

Case 3. \; $\Omega(n)=2$, $n=p^2$, where $p$ is a prime. Then
\begin{align*}
&\varphi^2(n)\psi^2(n)+\varphi^2(n)\sigma^2(n)+\sigma^2(n)\psi^2(n)\\
&=p^4(p^2-1)^2+p^2(p^3-1)^2+(p^2+p)^2(p^2+p+1)^2\\
&=3p^8+4p^7+6p^6+8p^5+9p^4+4p^3+2p^2\\
&>3p^8+4p^6+2p^4+4p^2+3\\
&>3n^4+4n^3+2n^2+4n+3\,.
\end{align*}
Let us assume that \eqref{lowerbound3} is true for every natural number $n$ with
$\Omega(n)=m$ for some natural number $m\geq2$.
Let $p$ be a prime number. Then $\Omega(np)=\Omega(n)+1$.

Case A.  \;  $p\nmid n$. Using that
\begin{equation}\label{psisigma}
\psi(n)\geq n+1\,, \quad \sigma(n)\geq n+1
\end{equation}
we derive
\begin{align*}
&\varphi^2(np)\psi^2(np)+\varphi^2(np)\sigma^2(np)+\sigma^2(np)\psi^2(np)\\
&=\varphi^2(n)\psi(n)^2(p^2-1)^2+\varphi^2(n)\sigma^2(n)(p^2-1)^2+\psi^2(n)\sigma^2(n)(p+1)^4\\
&=(p^2-1)^2\big[\varphi^2(n)\psi^2(n)+\varphi^2(n)\sigma^2(n)+\sigma^2(n)\psi^2(n)\big]+4p(p+1)^2\psi(n)\sigma(n)\\
&\geq(p^2-1)^2(3n^4+4n^3+2n^2+4n+3)+4p(p+1)^2(n+1)^4\\
&=3n^4p^4+16n^3p^3+44n^2p^2+16np+3+4n^4p^3+4n^3p^4+2n^4p^2+2n^2p^4\\
&+4n^4p+24n^3p^2+24n^2p^3+4np^4+3n^4+16n^3p+16np^3+3p^4\\
&+4n^3+24n^2p+24np^2+4p^3+2n^2+2p^2+4n+4p\\
&>3n^4p^4+4n^3p^3+2n^2p^2+4np+3\,.
\end{align*}

Case B.  \;  $p\,|\,n$. From \eqref{varphipsisigma} we establish
\begin{align*}
&\varphi^2(np)\psi^2(np)+\varphi^2(np)\sigma^2(np)+\sigma^2(np)\psi^2(np)\\
&>p^4\big[\varphi^2(n)\psi^2(n)+\varphi^2(n)\sigma^2(n)+\sigma^2(n)\psi^2(n)\big]\\
&\geq p^4(3n^4+4n^3+2n^2+4n+3)\\
&=3n^4p^4+4n^3p^3+2n^2p^2+4np+3+4n^3p^3(p-1)\\
&+2n^2p^2(p^2-1)+4np(p^3-1)+3(p^4-1)\\
&>3n^4p^4+4n^3p^3+2n^2p^2+4np+3\,.
\end{align*}
This completes the proof of Theorem 3.

\section{Proof of Theorem 4}
\indent

Consider several cases.

Case 1. \; $\Omega(n)=1$. Bearing in mind that $n$ is a prime number we write
\begin{align*}
&\varphi^2(n)\big(\psi(n)+\sigma(n)\big)+\psi^2(n)\big(\varphi(n)+\sigma(n)\big)+\sigma^2(n)\big(\varphi(n)+\psi(n)\big)\\
&=2(n-1)^2(n+1)+2(n+1)^2(n-1)+2(n+1)^3\\
&=6n^3+6n^2+2n+2\,.
\end{align*}

Case 2. \; $\Omega(n)=2$, $n=pq$, where $p$ and $q$ are distinct primes. Then
\begin{align*}
&\varphi^2(n)\big(\psi(n)+\sigma(n)\big)+\psi^2(n)\big(\varphi(n)+\sigma(n)\big)+\sigma^2(n)\big(\varphi(n)+\psi(n)\big)\\
&=2(p-1)^2(q-1)^2(p+1)(q+1)+2(p+1)^2(q+1)^2(p-1)(q-1)+2(p+1)^3(q+1)^3\\
&=6p^3q^3+22p^2q^2+22pq+6+6p^3q^2+6p^2q^3+2p^3q+2pq^3\\
&+2p^3+18p^2q+18pq^2+2q^3+2p^2+2q^2+6p+6q\\
&>6p^3q^3+6p^2q^2+2pq+2\\
&=6n^3+6n^2+2n+2\,.
\end{align*}

Case 3. \; $\Omega(n)=2$, $n=p^2$, where $p$ is a prime. Then
\begin{align*}
&\varphi^2(n)\big(\psi(n)+\sigma(n)\big)+\psi^2(n)\big(\varphi(n)+\sigma(n)\big)+\sigma^2(n)\big(\varphi(n)+\psi(n)\big)\\
&=p^3(p-1)^2(p+1)+p^3(p+1)^2(p-1)+p^2(p-1)^2(p^2+p+1)\\
&+p^2(p+1)^2(p^2+p+1)+p(p-1)(p^2+p+1)^2+p(p+1)(p^2+p+1)^2\\
&=6p^6+6p^5+8p^4+6p^3+4p^2\\
&>6p^6+6p^4+2p^2+2\\
&=6n^3+6n^2+2n+2\,.
\end{align*}
Now we assume that \eqref{lowerbound4} is true for every natural number $n$ with
$\Omega(n)=m$ for some natural number $m\geq2$.
Let $p$ be a prime number. Then $\Omega(np)=\Omega(n)+1$.

Case A.  \;  $p\nmid n$. Now \eqref{psisigma}, Lemma \ref{varphipsi2n} and Lemma \ref{varphisigma2n} imply
\begin{align*}
&\varphi^2(np)\big(\psi(np)+\sigma(np)\big)+\psi^2(np)\big(\varphi(np)+\sigma(np)\big)+\sigma^2(np)\big(\varphi(np)+\psi(np)\big)\\
&=\varphi^2(n)\psi(n)(p-1)^2(p+1)+\varphi^2(n)\sigma(n)(p-1)^2(p+1)+\psi^2(n)\varphi(n)(p+1)^2(p-1)\\
&+\psi^2(n)\sigma(n)(p+1)^3+\sigma^2(n)\varphi(n)(p+1)^2(p-1)+\sigma^2(n)\psi(n)(p+1)^3\\
&=(p-1)^2(p+1)\big[\varphi^2(n)\big(\psi(n)+\sigma(n)\big)+\psi^2(n)\big(\varphi(n)+\sigma(n)\big)+\sigma^2(n)\big(\varphi(n)+\psi(n)\big)\big]\\
&+(2p^2-2)\big[\psi^2(n)\varphi(n)+\sigma^2(n)\varphi(n)\big]+(4p^2+4p)\big[\psi^2(n)\sigma(n)+\sigma^2(n)\psi(n)\big]\\
&\geq(p-1)^2(p+1)(6n^3+6n^2+2n+2)+(2p^2+4p+2)\big[\psi^2(n)\sigma(n)+\sigma^2(n)\psi(n)\big]\\
&+(2p^2-2)\big[\psi^2(n)\big(\varphi(n)+\sigma(n)\big)+\sigma^2(n)\big(\varphi(n)+\psi(n)\big)\big]\\
&\geq(6n^3+6n^2+2n+2)(p-1)^2(p+1)+2(n+1)^3(2p^2+4p+2)+4n(n+1)^2(2p^2-2)\\
&=6n^3p^3+22n^2p^2+22np+6+6n^3p^2+6n^2p^3+2n^3p+2np^3+2n^3+18n^2p+18np^2\\
&+2p^3+2n^2+2p^2+6n+6p\\
&>6n^3p^3+6n^2p^2+2np+2\,.
\end{align*}

Case B.  \;  $p\,|\,n$. Using \eqref{varphipsisigma} we obtain
\begin{align*}
&\varphi^2(np)\big(\psi(np)+\sigma(np)\big)+\psi^2(np)\big(\varphi(np)+\sigma(np)\big)+\sigma^2(np)\big(\varphi(np)+\psi(np)\big)\\
&>p^3\big[\varphi^2(n)\big(\psi(n)+\sigma(n)\big)+\psi^2(n)\big(\varphi(n)+\sigma(n)\big)+\sigma^2(n)\big(\varphi(n)+\psi(n)\big)\big]\\
&\geq p^3(6n^3+6n^2+2n+2)\\
&=6n^3p^3+6n^2p^2+2np+2+6n^2p^2(p-1)+2np(p^2-1)+2(p^3-1)\\
&>6n^3p^3+6n^2p^2+2np+2\,.
\end{align*}
This completes the proof of Theorem 4.

\section{Proof of Theorem 5}
\indent

Consider several cases.

Case 1. \; $\Omega(n)=1$. Bearing in mind that $n$ is a prime number we write
\begin{align*}
&\varphi^3(n)\big(\psi(n)+\sigma(n)\big)+\psi^3(n)\big(\varphi(n)+\sigma(n)\big)+\sigma^3(n)\big(\varphi(n)+\psi(n)\big)\\
&=2(n-1)^3(n+1)+2(n+1)^3(n-1)+2(n+1)^4\\
&=6n^4+8n^3+12n^2+8n-2\,.
\end{align*}

Case 2. \; $\Omega(n)=2$, $n=pq$, where $p$ and $q$ are distinct primes. Then
\begin{align*}
&\varphi^3(n)\big(\psi(n)+\sigma(n)\big)+\psi^3(n)\big(\varphi(n)+\sigma(n)\big)+\sigma^3(n)\big(\varphi(n)+\psi(n)\big)\\
&=2(p-1)^3(q-1)^3(p+1)(q+1)+2(p+1)^3(q+1)^3(p-1)(q-1)+2(p+1)^4(q+1)^4\\
&=6p^4q^4+48p^3q^3+72p^2q^2+48pq+6+8p^4q^3+8p^3q^4+12p^4q^2+12p^2q^4\\
&+p^4(8q-2)+48p^3q^2+48p^2q^3+q^4(8p-2)+16p^3q+16pq^3+8p^3+48p^2q+48pq^2\\
&+8q^3+12p^2+12q^2+8p+8q\\
&>6p^4q^4+8p^3q^3+12p^2q^2+8pq-2\\
&=6n^4+8n^3+12n^2+8n-2\,.
\end{align*}

Case 3. \; $\Omega(n)=2$, $n=p^2$, where $p$ is a prime. Then
\begin{align*}
&\varphi^3(n)\big(\psi(n)+\sigma(n)\big)+\psi^3(n)\big(\varphi(n)+\sigma(n)\big)+\sigma^3(n)\big(\varphi(n)+\psi(n)\big)\\
&=p^4(p-1)^3(p+1)+p^4(p+1)^3(p-1)+p^3(p-1)^3(p^2+p+1)\\
&+p^3(p+1)^3(p^2+p+1)+p(p-1)(p^2+p+1)^3+p(p+1)(p^2+p+1)^3\\
&=6p^8+8p^7+20p^6+20p^5+16p^4+6p^3+2p^2\\
&>6p^8+8p^6+12p^4+8p^2-2\\
&=6n^4+8n^3+12n^2+8n-2\,.
\end{align*}
Now we assume that \eqref{lowerbound5} is true for every natural number $n$ with
$\Omega(n)=m$ for some natural number $m\geq2$.
Let $p$ be a prime number. Then $\Omega(np)=\Omega(n)+1$.

Case A.  \;  $p\nmid n$. Now \eqref{psisigma}, Lemma \ref{varphipsi2n} and Lemma \ref{varphisigma2n} yield
\begin{align*}
&\varphi^3(np)\big(\psi(np)+\sigma(np)\big)+\psi^3(np)\big(\varphi(np)+\sigma(np)\big)+\sigma^3(np)\big(\varphi(np)+\psi(np)\big)\\
&=\varphi^3(n)\psi(n)(p-1)^3(p+1)+\varphi^3(n)\sigma(n)(p-1)^3(p+1)+\psi^3(n)\varphi(n)(p+1)^3(p-1)\\
&+\psi^3(n)\sigma(n)(p+1)^4+\sigma^3(n)\varphi(n)(p+1)^3(p-1)+\sigma^3(n)\psi(n)(p+1)^4\\
&=(p-1)^3(p+1)\big[\varphi^3(n)\big(\psi(n)+\sigma(n)\big)+\psi^3(n)\big(\varphi(n)+\sigma(n)\big)+\sigma^3(n)\big(\varphi(n)+\psi(n)\big)\big]\\
&+(4p^3-4p)\big[\psi^3(n)\varphi(n)+\sigma^3(n)\varphi(n)\big]+(6p^3+6p^2+2p+2)\big[\psi^3(n)\sigma(n)+\sigma^3(n)\psi(n)\big]\\
&\geq(p-1)^3(p+1)(6n^4+8n^3+12n^2+8n-2)\\
&+(4p^3-4p)\big[\psi^3(n)\big(\varphi(n)+\sigma(n)\big)+\sigma^3(n)\big(\varphi(n)+\psi(n)\big)\big]\\
&+(2p^3+6p^2+6p+2)\big[\psi^3(n)\sigma(n)+\sigma^3(n)\psi(n)\big]\\
&\geq(6n^4+8n^3+12n^2+8n-2)(p-1)^3(p+1)\\
&+4n(n+1)^3(4p^3-4p)+2(n+1)^4(2p^3+6p^2+6p+2)\\
&=6n^4p^4+48n^3p^3+72n^2p^2+48np+6+n^4(8p^3-2)+8n^3p^4+12n^4p^2\\
&+p^4(12n^2-2)+8n^4p+48n^3p^2+48n^2p^3+8np^4+16n^3p+16np^3+8n^3+48n^2p\\
&+48np^2+8p^3+12n^2+12p^2+8n+8p\\
&>6n^4p^4+8n^3p^3+12n^2p^2+8np-2\,.
\end{align*}

Case B.  \;  $p\,|\,n$. Using \eqref{varphipsisigma} we deduce
\begin{align*}
&\varphi^3(np)\big(\psi(np)+\sigma(np)\big)+\psi^3(np)\big(\varphi(np)+\sigma(np)\big)+\sigma^3(np)\big(\varphi(np)+\psi(np)\big)\\
&>p^4\big[\varphi^3(n)\big(\psi(n)+\sigma(n)\big)+\psi^3(n)\big(\varphi(n)+\sigma(n)\big)+\sigma^3(n)\big(\varphi(n)+\psi(n)\big)\big]\\
&\geq p^4(6n^4+8n^3+12n^2+8n-2)\\
&=6n^4p^4+8n^3p^3+12n^2p^2+8np+8n^3p^3(p-1)\\
&+12n^2p^2(p^2-1)+4np(p^3-2)+2p^4(2n-1)\\
&>6n^4p^4+8n^3p^3+12n^2p^2+8np-2\,.
\end{align*}
This completes the proof of Theorem 5.


\begin{thebibliography}{0}

\bibitem{Atanassov2011} Atanassov, K. (2011). Note on $\varphi$, $\psi$ and $\sigma$-functions. Part 3.
{\it Notes on Number Theory and Discrete Mathematics}, 17(3), 13 -- 14.

\bibitem{Atanassov2013} Atanassov, K. (2013). Note on $\varphi$, $\psi$ and $\sigma$-functions. Part 6.
{\it Notes on Number Theory and Discrete Mathematics}, 19(1), 22 -- 24.

\bibitem{Dimitrov} Dimitrov, S. (2023). Lower bounds on expressions dependent on functions $\varphi(n)$, $\psi(n)$ and $\sigma(n)$.
{\it Notes on Number Theory and Discrete Mathematics}, 29(4), 22 -- 24.

\bibitem{Sandor1990} S\'{a}ndor, J., $\&$ T\'{o}th., L. (1990). On certain number-theoretic inequalities.
{\it Fib. Quart.}, 28(3), 255 -- 258.

\bibitem{SandorHandbook} S\'{a}ndor, J., Mitrinovi\'{c}, D. S., $\&$ Crstici, B. (2006). {\it Handbook of number theory I}.
Springer.

\bibitem{Sandor2014} S\'{a}ndor, J. (2014). On certain inequalities for $\varphi$, $\psi$, $\sigma$ and related functions.
{\it Notes on Number Theory and Discrete Mathematics}, 20(2), 52 -- 60.

\bibitem{Sandor2019} S\'{a}ndor, J., $\&$ Atanassov, K. (2019). Inequalities between the arithmetic functions
$\varphi$, $\psi$ and $\sigma$. Part 2.
{\it Notes on Number Theory and Discrete Mathematics}, 25(2), 30 -- 35.

\end{thebibliography}
\end{document}